\documentclass{article}%
\usepackage{amsmath}
\usepackage{amsfonts}
\usepackage{amssymb}
\usepackage{graphicx}%
\providecommand{\U}[1]{\protect \rule{.1in}{.1in}}
\newtheorem{theorem}{Theorem}

\newtheorem{remark}[theorem]{Remark}

\newenvironment{proof}[1][Proof]{\noindent \textbf{#1.} }{\  \rule{0.5em}{0.5em}}
\begin{document}

\title{Extension of some Cheney-Sharma type operators to a triangle with one curved
side }
\author{Teodora C\u{a}tina\c{s} \thanks{Babe\c{s}-Bolyai University, Faculty of
Mathematics and Computer Science, Str. M. Kog\u{a}lniceanu Nr. 1, RO-400084
Cluj-Napoca, Romania, e-mail: tcatinas@math.ubbcluj.ro}}
\date{}
\maketitle

\begin{abstract}
We extend some Cheney-Sharma type operators to a triangle with one curved
side. We construct their product and Boolean sum, we study their interpolation
properties, the orders of accuracy and we give different expressions of the
corresponding remainders. We also give some illustrative examples.

\end{abstract}

\newline modulus of continuity, degree of exactness, the Peano's theorem,
error evaluation.


\section{Introduction}

In order to match all the boundary information on a curved domain (as
Dirichlet, Neumann or Robin boundary conditions for differential equation
problems), there were considered interpolation operators on domains with
curved sides (see, e.g., \cite{BarGreJAT75}, \cite{Ber89}, \cite{BlaCatCom09}%
-\cite{Cat17}, \cite{ComCat09}, \cite{ComCat10}, \cite{MarMcL74},
\cite{MarMit73}, \cite{MitMcL74}).

The aim of this paper is to construct some Cheney-Sharma type operators on a
triangle with one curved side and to study the interpolation properties, the
orders of accuracy and the remainders of the corresponding approximation formulas.

Using the interpolation properties of such operators, blending function
interpolants can be constructed, that exactly match the function on some sides
of the given region. Important applications of these blending functions are in
computer aided geometric design (see, e.g., \cite{Bar76}-\cite{BarBirGor73},
\cite{Bar76b}, \cite{Sch76}), in finite element method for differential
equations problems (see, e.g., \cite{Bar76}, \cite{GorHall72}, \cite{Gor74}, \cite{MarMit73}, \cite{MArMit78},
\cite{Zla73}) or for construction of surfaces which satisfy some given
conditions (see, e.g., \cite{CatBlaCom13}, \cite{ComGan86}, \cite{ComGanTam91}).

\section{Univariate operators}

Let $m\in \mathbb{N}$ and $\beta$ be a nonnegative parameter. The Cheney-Sharma
operators of second kind $Q_{m}:C[0,1]\rightarrow C[0,1]$, introduced in
\cite{CheSar64}, are given by
\begin{equation}
(Q_{m}f)(x)=%
{\textstyle \sum \limits_{i=0}^{m}}
{q}_{m,i}(x)f(\tfrac{k}{m}), \label{Chorig}%
\end{equation}%
\[
{q}_{m,i}(x)={\tbinom{m}{i}}\frac{x(x+i\beta)^{i-1}(1-x)[1-x+(m-i)\beta
]^{m-i-1}}{(1+m\beta)^{m-1}}.
\]
We recall some results regarding these Cheney-Sharma type operators.

\begin{remark}
\label{proprChShorig}1) Notice that for $\beta=0$, the operator $Q_{m}$
becomes the Bernstein operator.

2) In \cite{StaCis97}, there have been proved that the Cheney-Sharma operator
$Q_{m}$ interpolates a given function at the endpoints of the interval.

3) In \cite{CheSar64} and \cite{StaCis97}, there have been proved that the
Cheney-Sharma operator $Q_{m}$ reproduces the constant and the linear
functions, so its degree of exactness is $1$ (denoted $\operatorname*{dex}%
(Q_{m})=1)$.

4) In \cite{CheSar64} it is given the following result
\begin{align}
(Q_{m}e_{2})(x)=  &  x(1+m\beta)^{1-m}[S(2,m-2,x+2\beta,1-x)\label{qe2}\\
&  -(m-2)\beta S(2,m-3,x+2\beta,1-x+\beta)],\nonumber
\end{align}
where $e_{i}(x)=x^{i},$ $i\in \mathbb{N,}$ and
\begin{equation}
S(j,m,x,y)=\sum_{k=0}^{m}{\binom{m}{k}}(x+k\beta)^{k+j-1}[y+(m-k)\beta]^{m-k},
\label{S}%
\end{equation}
$j=\overline{0,m}$, $m\in \mathbb{N}$, $x,y\in \lbrack0,1]$, $\beta>0$.
\end{remark}

We consider the standard triangle $\tilde{T}_{h}$ (see Figure 1), with
vertices $V_{1}=(0,h),$ $V_{2}=(h,0)$ and $V_{3}=(0,0),$ with two straight
sides $\Gamma_{1},$ $\Gamma_{2},$ along the coordinate axes, and with the
third side $\Gamma_{3}$ (opposite to the vertex $V_{3}$) defined by the
one-to-one functions $f$ and $g,$ where $g$ is the inverse of the function
$f,$ i.e., $y=f(x)$ and $x=g(y)$, with $f(0)=g(0)=h,\ $for $h>0$. Also, we
have $f(x)\leq h$ and $g(y)\leq h,$ for $x,y\in \left[  0,h\right]  .$%

\begin{figure}
[ptb]
\begin{center}
\includegraphics[
height=1.8005in,
width=2.0712in
]%
{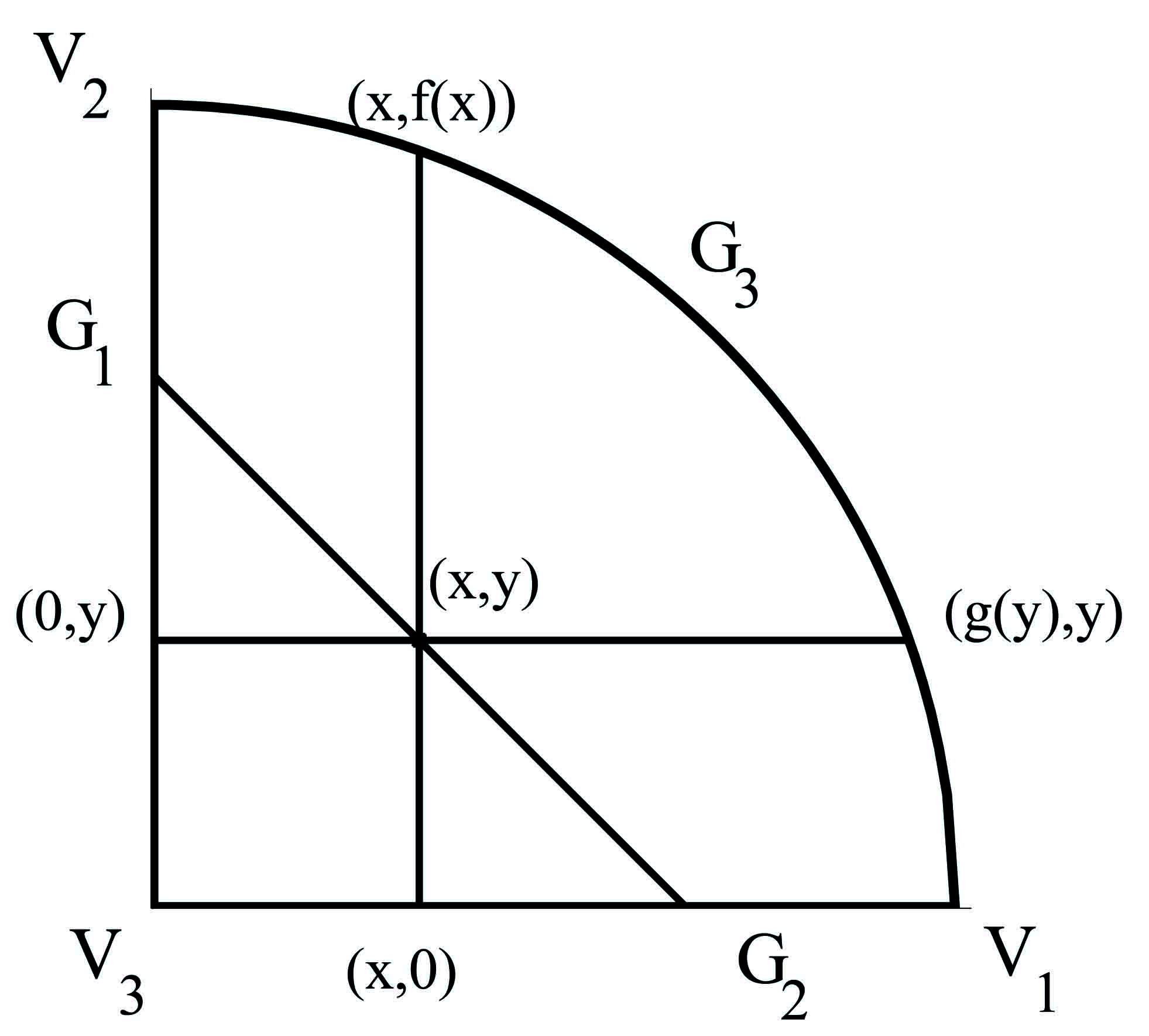}%
\caption{Triangle $\tilde{T}_{h}.$}%
\end{center}
\end{figure}

For $m,n\in \mathbb{N},$ $\beta,b\in \mathbb{R}_{+},$ we consider the following
extensions of the Cheney-Sharma operator given in (\ref{Chorig}):%
\begin{align}
(Q_{m}^{x}F)(x,y)  &  =%
{\textstyle \sum \limits_{i=0}^{m}}
{q}_{m,i}(x,y)F\left(  i\tfrac{g(y)}{m},y\right)  ,\label{Chcurb}\\
(Q_{n}^{y}F)(x,y)  &  =%
{\textstyle \sum \limits_{j=0}^{n}}
{q}_{n,j}(x,y)F\left(  x,j\tfrac{f(x)}{n}\right)  ,\nonumber
\end{align}
with
\begin{align*}
q_{m,i}\left(  x,y\right)   &  ={\tbinom{m}{i}}\tfrac{\tfrac{x}{g(y)}%
(\tfrac{x}{g(y)}+i\beta)^{i-1}(1-\tfrac{x}{g(y)})[1-\tfrac{x}{g(y)}%
+(m-i)\beta]^{m-i-1}}{(1+m\beta)^{m-1}},\\
q_{n,j}\left(  x,y\right)   &  ={\tbinom{n}{j}}\tfrac{\tfrac{y}{f(x)}%
(\tfrac{y}{f(x)}+jb)^{j-1}(1-\tfrac{y}{f(x)})[1-\tfrac{y}{f(x)}%
+(n-j)b]^{n-j-1}}{(1+nb)^{n-1}},
\end{align*}
where%
\[
\Delta_{m}^{x}=\left \{  \left.  i\tfrac{g(y)}{m}\right \vert \ i=\overline
{0,m}\right \}  \  \text{and }\Delta_{n}^{y}=\left \{  \left.  j\tfrac{f(x)}%
{n}\right \vert \ j=\overline{0,n}\right \}
\]
are uniform partitions of the intervals $[0,g(y)]$ and $[0,f(x)].$

\begin{remark}
As the Cheney-Sharma operator of second kind interpolates a given function at
the endpoints of the interval, we may use the operators $Q_{m}^{x}$ and
$Q_{n}^{y}$ as interpolation operators.
\end{remark}

\begin{theorem}
\label{proprQ}If $F$ is a real-valued function defined on $\widetilde{T}%
_{h}\ $then
\end{theorem}

\begin{enumerate}
\item[\textrm{(i)}] $Q_{m}^{x}F=F\ $on\ $\Gamma_{1}\cup \Gamma_{3},$

\item[\textrm{(ii)}] $Q_{n}^{y}F=F\ $on\ $\Gamma_{2}\cup \Gamma_{3}.$
\end{enumerate}

\begin{proof}
(i) We may write%
\begin{align}
(Q_{m}^{x}F)(x,y)=  &  \tfrac{1}{(1+m\beta)^{m-1}}\left \{  (1-\tfrac{x}%
{g(y)})[1-\tfrac{x}{g(y)}+m\beta]^{m-1}F\left(  0,y\right)  \right.
\label{chforma2}\\
&  +\tfrac{x}{g(y)}(1-\tfrac{x}{g(y)})%
{\textstyle \sum \limits_{i=1}^{m-1}}
{\tbinom{m}{i}}(\tfrac{x}{g(y)}+i\beta)^{i-1}\nonumber \\
&  \cdot \lbrack1-\tfrac{x}{g(y)}+(m-i)\beta]^{m-i-1}F\left(  i\tfrac{g(y)}%
{m},y\right) \nonumber \\
&  \left.  +\tfrac{x}{g(y)}(\tfrac{x}{g(y)}+m\beta)^{m-1}F\left(
g(y),y\right)  \right \}  .\nonumber
\end{align}
Considering (\ref{chforma2}), we may easily prove that
\begin{align*}
&  (Q_{m}^{x}F)(0,y) =F(0,y),\\
&  (Q_{m}^{x}F)(g(y),y) =F(g(y),y).
\end{align*}
(ii) Similarly, writing
\begin{align*}
(Q_{n}^{y}F)(x,y)=  &  \tfrac{1}{(1+nb)^{n-1}}\left \{  (1-\tfrac{y}%
{f(x)})[1-\tfrac{y}{f(x)}+nb]^{n-1}F\left(  x,0\right)  \right. \\
&  +\tfrac{y}{f(x)}(1-\tfrac{y}{f(x)})%
{\textstyle \sum \limits_{j=1}^{n-1}}
{\tbinom{n}{j}}(\tfrac{y}{f(x)}+jb)^{j-1}\\
&  \cdot \lbrack1-\tfrac{y}{f(x)}+(n-j)b]^{n-j-1}F\left(  x,j\tfrac{f(x)}%
{n}\right) \\
&  \left.  +\tfrac{y}{f(x)}(\tfrac{y}{f(x)}+nb)^{n-1}F\left(  x,f(x)\right)
\right \}  ,
\end{align*}
we get that
\begin{align*}
&  (Q_{n}^{y}F)(x,0)=F(x,0),\\
&  (Q_{n}^{y}F)(x,f(x))=F(x,f(x)).
\end{align*}

\end{proof}

\begin{theorem}
\label{dex}The operators $Q_{m}^{x}$ and $Q_{n}^{y}$ have the following orders
of accuracy:
\end{theorem}

\begin{enumerate}
\item[(i)] $\left(  Q_{m}^{x}e_{ij}\right)  \left(  x,y\right)  =x^{i}%
y^{j},\  \ i=0,1;$ $j\in \mathbb{N};$

\item[(ii)] $\left(  Q_{n}^{y}e_{ij}\right)  \left(  x,y\right)  =x^{i}%
y^{j},\ $ $i\in \mathbb{N};$ $j=0,1,$ where $e_{ij}\left(  x,y\right)
=x^{i}y^{j},\ $ $i,j\in \mathbb{N.}$
\end{enumerate}

\begin{proof}
(i) We have
\[
(Q_{m}^{x}e_{ij})(x,y)=y^{j}%
{\textstyle \sum \limits_{i=0}^{m}}
{q}_{m,i}(x,y)[i\tfrac{g(y)}{m}]^{i},
\]
and having the degree of exactness of the univariate Cheney-Sharma operator
equal to $1$ (see Remark \ref{proprChShorig}), the result follows.

Property (ii) is proved in the same way.
\end{proof}

We consider the approximation formula%
\[
F=Q_{m}^{x}F+R_{m}^{x}F,
\]
where $R_{m}^{x}F$ denotes the approximation error.

\begin{theorem}
\label{Theorema 2.}If $F(\cdot,y)\in C[0,g(y)]$ then we have%
\begin{equation}
\big \vert \left(  R_{m}^{x}F\right)  (x,y)\big \vert \leq(1+\tfrac{1}{\delta
}\sqrt{A_{m}-x^{2}})\omega(F(\cdot,y);\delta),\  \  \  \  \forall \delta>0,
\label{est1}%
\end{equation}
where $\omega(F(\cdot,y);\delta)$ is the modulus of continuity and
$A_{m}=x(1+m\beta)^{1-m}[S(2,m-2,x+2\beta,1-x)-(m-2)\beta S(2,m-3,x+2\beta
,1-x+\beta)],$ with $S$ given in (\ref{S}).
\end{theorem}

\begin{proof}
By Theorem \ref{dex} we have that $\operatorname*{dex}(Q_{m}^{x})=1,$ thus we
may apply the following property of linear operators (see, for example,
\cite{Agr00}):
\[
\big \vert(Q_{m}^{x}F)(x,y)-F(x,y)\big \vert \leq \lbrack1+\delta^{-1}%
\sqrt{(Q_{m}^{x}e_{20})(x,y)-x^{2}}]\omega(F(\cdot,y);\delta),\  \  \forall
\delta>0,
\]
and taking into account (\ref{qe2}), we get (\ref{est1}).
\end{proof}

\begin{theorem}
\label{Theorema 3.}If $F(\cdot,y)\in C^{2}[0,g(y)]$ then
\begin{align}
(R_{m}^{x}F)(x,y)=  &  \tfrac{1}{2}F^{(2,0)}(\xi,y)\{x^{2}-x(1+m\beta
)^{1-m}[S(2,m-2,x+2\beta,1-x)\label{estRP1}\\
&  -(m-2)\beta S(2,m-3,x+2\beta,1-x+\beta)]\},\nonumber
\end{align}
for $\xi \in \lbrack0,g(y)]$ and $\beta>0.$
\end{theorem}

\begin{proof}
Taking into account that $\operatorname*{dex}(Q_{m}^{x})=1,$ by Theorem
\ref{dex}, and applying the Peano's theorem (see, e.g., \cite{Sar63}), it
follows%
\[
(R_{m}^{x}F)(x,y)=\int_{0}^{g(y)}K_{20}(x,y;s)F^{(2,0)}(s,y)ds,
\]
where%
\[
K_{20}(x,y;s)=(x-s)_{+}-\sum_{i=0}^{m}q_{m,i}(x,y)\big(i\tfrac{g(y)}%
{m}-s\big)_{+}.
\]
For a given $\nu \in \{1,...,m\}$ one denotes by $K_{20}^{\nu}(x,y;\cdot)$ the
restriction of the kernel $K_{20}(x,y;\cdot)$ to the interval $\left[
(\nu-1)\frac{g(y)}{m},\nu \frac{g(y)}{m}\right]  ,$ i.e.,%
\[
K_{20}^{\nu}(x,y;\nu)=(x-s)_{+}-\sum_{i=\nu}^{m}q_{m,i}(x,y)\left(
i\tfrac{g(y)}{m}-s\right)  ,
\]
whence,%
\[
K_{20}^{\nu}(x,y;s)=\left \{
\begin{array}
[c]{ll}%
x-s-\sum \limits_{i=\nu}^{m}q_{m,i}(x,y)\left(  i\frac{g(y)}{m}-s\right)  , &
\ s<x\\[3mm]%
-\sum \limits_{i=\nu}^{m}q_{m,i}(x,y)\left(  i\frac{g(y)}{m}-s\right)  , &
\ s\geq x.
\end{array}
\right.
\]
It follows that $K_{20}^{\nu}(x,y;s)\leq0,\ $for $s\geq x.$

For $s<x$ we have%
\begin{align*}
K_{20}^{\nu}(x,y;s)=  &  x-s-\sum_{i=0}^{m}q_{m,i}(x,y)\left[  i\tfrac
{g(y)}{m}-s\right] \\
&  +\sum_{i=0}^{\nu-1}q_{m,i}(x,y)\left[  i\tfrac{g(y)}{m}-s\right]  .
\end{align*}
Applying Theorem \ref{dex}, we get%
\begin{align*}
\sum_{i=0}^{m}q_{m,i}(x,y)\left[  i\tfrac{g(y)}{m}-s\right]   &  =(Q_{m}%
^{x}e_{10})(x,y)-s(Q_{m}^{x}e_{00})(x,y)\\
&  =x-s,
\end{align*}
whence it follows that%
\[
K_{20}^{\nu}(x,y;s)=\sum_{i=0}^{\nu-1}q_{m,i}(x,y)\left[  i\tfrac{g(y)}%
{m}-s\right]  \leq0.
\]
So, $K_{20}^{\nu}(x,y;\cdot)\leq0,$ for any $\nu \in \{1,...,m\},$ i.e.,
$K_{20}(x,y;s)\leq0,$ for $s\in \lbrack0,g(y)].$

By the Mean Value Theorem, one obtains%
\[
(R_{m}^{x}F)(x,y)=F^{(2,0)}(\xi,y)\int_{0}^{g(y)}K_{20}(x,y;s)ds,\  \  \text{for
}0\leq \xi \leq g(y),
\]
with%
\[
\int_{0}^{g(y)}K_{20}(x,y;s)ds=\tfrac{1}{2}[x^{2}-(Q_{m}^{x}e_{20})(x,y)],
\]
and using (\ref{qe2}) we get (\ref{estRP1}).
\end{proof}

\begin{remark}
\label{Remark 2.} Analogous results with the ones in Theorems
\ref{Theorema 2.} and \ref{Theorema 3.} could be obtained for the remainder
$R_{n}^{y}F$ of the formula $F=Q_{n}^{y}F+R_{n}^{y}F.$
\end{remark}

\section{Product operators}

Let $P_{mn}^{1}=Q_{m}^{x}Q_{n}^{y},$ respectively, $P_{nm}^{2}=Q_{n}^{y}%
Q_{m}^{x}$ be the products of the operators $Q_{m}^{x}$ and $Q_{n}^{y}.$

We have%
\[
\left(  P_{mn}^{1}F\right)  \left(  x,y\right)  \!=\! \sum_{i=0}^{m}\sum
_{j=0}^{n}q_{m,i}\left(  x,y\right)  q_{n,j}\left(  i\tfrac{g(y)}{m},y\right)
F\Big(i\tfrac{g(y)}{m},j\tfrac{f(i\tfrac{g(y)}{m})}{n}\Big),
\]
respectively,%
\[
\left(  P_{nm}^{2}F\right)  \left(  x,y\right)  \!=\! \sum_{i=0}^{m}\sum
_{j=0}^{n}q_{m,i}\left(  x,j\tfrac{f(x)}{n}\right)  q_{n,j}\left(  x,y\right)
F\Big(i\tfrac{g(j\tfrac{f(x)}{n})}{m},j\tfrac{f(x)}{n}\Big).
\]

\begin{theorem}
\label{Theorema 4.} If $F$ is a real-valued function defined on $\widetilde
{T}_{h}$ then

\begin{enumerate}
\item[(i)] $(P_{mn}^{1}F)(V_{i})=F(V_{i}),\  \  \  \ i=1,...,3;$

$(P_{mn}^{1}F)(\Gamma_{3})=F(\Gamma_{3}),\ $

\item[(ii)] $(P_{nm}^{2}F)(V_{i})=F(V_{i}),\  \  \  \ i=1,...,3;$

$(P_{nm}^{2}F)(\Gamma_{3})=F(\Gamma_{3}),\ $
\end{enumerate}
\end{theorem}

\begin{proof}
By a straightforward computation, we get the following properties
\begin{align*}
&  (P_{mn}^{1}F)(x,0)=(Q_{m}^{x}F)(x,0),\\
&  (P_{mn}^{1}F)(0,y)=(Q_{n}^{y}F)(0,y),\\
&  (P_{mn}^{1}F)(x,f(x))=F(x,f(x)),\  \  \  \  \ x,y\in \lbrack0,h]
\end{align*}
and%
\begin{align*}
&  (P_{nm}^{2}F)(x,0)=(Q_{m}^{x}F)(x,0),\\
&  (P_{nm}^{2}F)(0,y)=(Q_{n}^{y}F)(0,y),\\
&  (P_{nm}^{2}F)(g(y),y)=F(g(y),y),\  \  \  \  \  \ x,y\in \lbrack0,h],
\end{align*}
and, taking into account Theorem \ref{proprQ}, they imply (i) and (ii).
\end{proof}

We consider the following approximation formula%
\[
F=P_{mn}^{1}F+R_{mn}^{^{P^{1}}}F,
\]
where $R_{mn}^{^{P^{1}}}$ is the corresponding remainder operator.

\begin{theorem}
\label{Theorem 5.} If $F\in C(\widetilde{T}_{h})$ then%
\begin{equation}
\left \vert (R_{mn}^{P^{1}}F)(x,y)\right \vert \leq(A_{m}+B_{n}-x^{2}%
-y^{2}+1)\omega(F;\tfrac{1}{\sqrt{A_{m}-x^{2}}},\tfrac{1}{\sqrt{B_{n}-y^{2}}%
}),\  \forall(x,y)\in \widetilde{T}_{h}, \label{restP1}%
\end{equation}
where
\begin{align}
A_{m}=  &  x(1+m\beta)^{1-m}[S(2,m-2,x+2\beta,1-x)\label{ab}\\
&  -(m-2)\beta S(2,m-3,x+2\beta,1-x+\beta)]\nonumber \\
B_{n}=  &  y(1+nb)^{1-n}[S(2,n-2,y+2b,1-y)-(n-2)bS(2,n-3,y+2b,1-y+\beta
)]\nonumber
\end{align}
and $\omega(F;\delta_{1},\delta_{2})$, with $\delta_{1}>0,$ $\delta_{2}>0,$ is
the bivariate modulus of continuity.
\end{theorem}

\begin{proof}
Using a basic property of the modulus of continuity we have%
\begin{align*}
\left \vert (R_{mn}^{P^{1}}F)(x,y)\right \vert \leq &  \bigg[\tfrac{1}%
{\delta_{1}}\sum_{i=0}^{m}\sum_{j=0}^{n}q_{m,i}(x,y)q_{n,j}\left(  \tfrac
{i}{m}g(y),y\right)  \left \vert x-\tfrac{i}{m}g(y)\right \vert \\
&  +\tfrac{1}{\delta_{2}}\sum_{i=0}^{m}\sum_{j=0}^{n}q_{m,i}(x,y)q_{n,j}%
\left(  \tfrac{i}{m}g(y),y\right)  \left \vert y-\tfrac{j}{n}f\left(  \tfrac
{i}{m}g(y)\right)  \right \vert \\
&  +\sum_{i=0}^{m}\sum_{j=0}^{n}q_{m,i}(x,y)q_{n,j}\left(  \tfrac{i}%
{m}g(y),y\right)  \bigg]\omega(F;\delta_{1},\delta_{2}),\  \  \forall \delta
_{1},\delta_{2}>0.
\end{align*}
Since%
\begin{align*}
&  \sum_{i=0}^{m}\sum_{j=0}^{n}p_{m,i}(x,y)q_{n,j}\left(  \tfrac{i}%
{m}g(y),y\right)  \left \vert x-\tfrac{i}{m}g(y)\right \vert \leq \sqrt
{(Q_{m}^{x}e_{20})(x,y)-x^{2}},\\
&  \sum_{i=0}^{m}\sum_{j=0}^{n}p_{m,i}(x,y)q_{n,j}\left(  \tfrac{i}%
{m}g(y),y\right)  \left \vert y-\tfrac{j}{n}f\big(\tfrac{i}{m}%
g(y)\big)\right \vert \leq \sqrt{(Q_{n}^{y}e_{02})(x,y)-y^{2}},\\
&  \sum_{i=0}^{m}\sum_{j=0}^{n}p_{m,i}(x,y)q_{n,j}\left(  \tfrac{i}%
{m}g(y),y\right)  =1,
\end{align*}
applying (\ref{qe2}), we get
\begin{align*}
&  \left \vert (R_{mn}^{P^{1}}F)(x,y)\right \vert \leq \left \{  \tfrac{1}%
{\delta_{1}}[x(1+m\beta)^{1-m}]^{\frac{1}{2}} \right. \\
&  \cdot \left \{  \lbrack S(2,m-2,x+2\beta,1-x)-(m-2)\beta S(2,m-3,x+2\beta
,1-x+\beta)]-x^{2}\right \}  ^{\frac{1}{2}}\\
&  +\tfrac{1}{\delta_{2}}[y(1+nb)^{1-n}]^{\frac{1}{2}}\\
&  \cdot \left \{  \lbrack S(2,n-2,y+2b,1-y)-(n-2)bS(2,n-3,y+2b,1-y+\beta
)]-y^{2}\right \}  ^{\frac{1}{2}}\\
&  \left.  +1\right \}  \omega(F;\delta_{1},\delta_{2}).
\end{align*}
Denoting
\begin{align*}
A_{m}  &  =x(1+m\beta)^{1-m}[S(2,m-2,x+2\beta,1-x)-(m-2)\beta S(2,m-3,x+2\beta
,1-x+\beta)]\\
B_{n}  &  =y(1+nb)^{1-n}[S(2,n-2,y+2b,1-y)-(n-2)bS(2,n-3,y+2b,1-y+\beta)]
\end{align*}
and, taking $\delta_{1}=\frac{1}{\sqrt{A_{m}-x^{2}}}$ and $\delta_{2}=\frac
{1}{\sqrt{B_{n}-y^{2}}}$, we get (\ref{restP1}).
\end{proof}

\section{Boolean sum operators}

We consider the Boolean sums of the operators $Q_{m}^{x}$ and $Q_{n}^{y}$,%
\begin{align*}
S_{mn}^{1}  &  :=Q_{m}^{x}\oplus Q_{n}^{y}=Q_{m}^{x}+Q_{n}^{y}-Q_{m}^{x}%
Q_{n}^{y},\\
S_{nm}^{2}  &  :=Q_{n}^{y}\oplus Q_{m}^{x}=Q_{n}^{y}+Q_{m}^{x}-Q_{n}^{y}%
Q_{m}^{x}.
\end{align*}

\begin{theorem}
If $F$ is a real-valued function defined on $\widetilde{T}_{h},$ then%
\begin{align*}
&  S_{mn}^{1}F\left \vert _{\partial \widetilde{T}_{h}}=F\right \vert
_{\partial \widetilde{T}_{h}},\\
&  S_{mn}^{2}F\left \vert _{\partial \widetilde{T}_{h}}=F\right \vert
_{\partial \widetilde{T}_{h}}.
\end{align*}

\end{theorem}

\begin{proof}
We have%
\begin{align*}
\left(  Q_{m}^{x}Q_{n}^{y}F\right)  \left(  x,0\right)  =  &  \left(
Q_{m}^{x}F\right)  \left(  x,0\right)  {,}\\
\left(  Q_{n}^{y}Q_{m}^{x}F\right)  \left(  0,y\right)  =  &  \left(
Q_{n}^{y}F\right)  \left(  0,y\right)  {,}\\
\left(  Q_{m}^{x}F\right)  \left(  x,h-x\right)  =  &  \left(  Q_{n}%
^{y}F\right)  (x,h-x)\\
&  =(P_{mn}^{1}F)(x,h-x)=(P_{nm}^{2}F)(x,h-x)=F(x,h-x),
\end{align*}
and, taking into account Theorem \ref{proprQ}, the conclusion follows.
\end{proof}

We consider the following approximation formula%
\[
F=S_{mn}^{1}F+R_{mn}^{^{S^{1}}}F,
\]
where $R_{mn}^{^{S^{1}}}$ is the corresponding remainder operator.

\begin{theorem}
If $F\in C(\widetilde{T}_{h})$ then%
\begin{align}
&  \big \vert(R_{mn}^{^{S^{1}}}F)(x,y)\big \vert \leq \label{estS}\\
&  \leq(1+A_{m}-x^{2})\omega(F(\cdot,y);\tfrac{1}{\sqrt{A_{m}-x^{2}}%
})+(1+B_{n}-y^{2})\omega(F(x,\cdot);\tfrac{1}{\sqrt{B_{n}-y^{2}}})\nonumber \\
&  +(A_{m}+B_{n}-x^{2}-y^{2}+1)\omega(F;\tfrac{1}{\sqrt{A_{m}-x^{2}}}%
,\tfrac{1}{\sqrt{B_{n}-y^{2}}}),\nonumber
\end{align}
with $A_{m}$ and $B_{n}$ given in (\ref{ab}).
\end{theorem}

\begin{proof}
The identity
\[
F-S_{mn}^{1}F=(F-Q_{m}^{x}F)+(F-Q_{n}^{y}F)-(F-P_{mn}^{1}F)
\]
implies that
\[
\big \vert(R_{mn}^{^{S^{1}}}F)(x,y)\big \vert \leq \big \vert(R_{m}%
^{x}F)(x,y)\big \vert+\big \vert(R_{n}^{y}F)(x,y)\big \vert+\big \vert(R_{mn}%
^{P^{1}}F)(x,y)\big \vert,
\]
and, applying Theorems \ref{Theorema 2.} and \ref{Theorem 5.}, we get
(\ref{estS}).
\end{proof}

\section{Numerical examples}

We consider the function:%
\[%
\begin{array}
[c]{ll}%
\text{Gentle:} & F(x,y)=\frac{1}{3}\exp[-\tfrac{81}{16}\left(  (x-0.5)^{2}%
+(y-0.5)^{2}\right)  ],
\end{array}
\]
generally used in the literature, (see, e.g., \cite{RenCli84}). In Figure 2 we
plot the graphs of $F,$ $Q_{m}^{x}F,$ $Q_{n}^{y}F,$ $P_{mn}^{1}F,$ $S_{mn}%
^{1}F$, on $\tilde{T}_{h},$ considering $h=1,m=5,$ $n=6$, $\beta=1$ and we can
see the good approximation properties.\newpage%
\begin{center}
\includegraphics[
height=1.3102in,
width=1.7288in
]%
{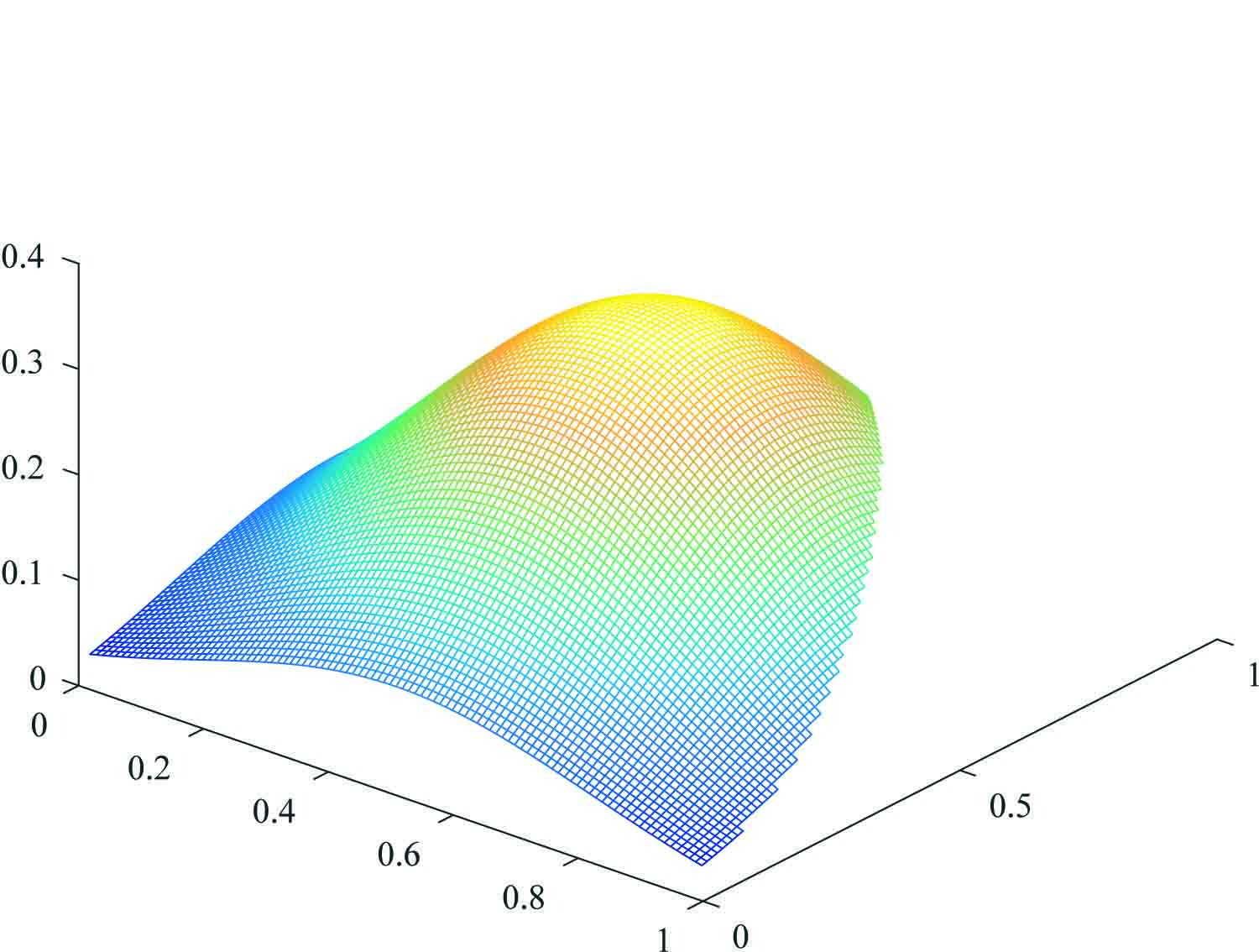}%
\\
$F$%
\end{center}
\  \  \  \  \  \  \  \  \  \  \  \  \  \ %

{\parbox[b]{1.7365in}{\begin{center}
\includegraphics[
height=1.2652in,
width=1.7365in
]%
{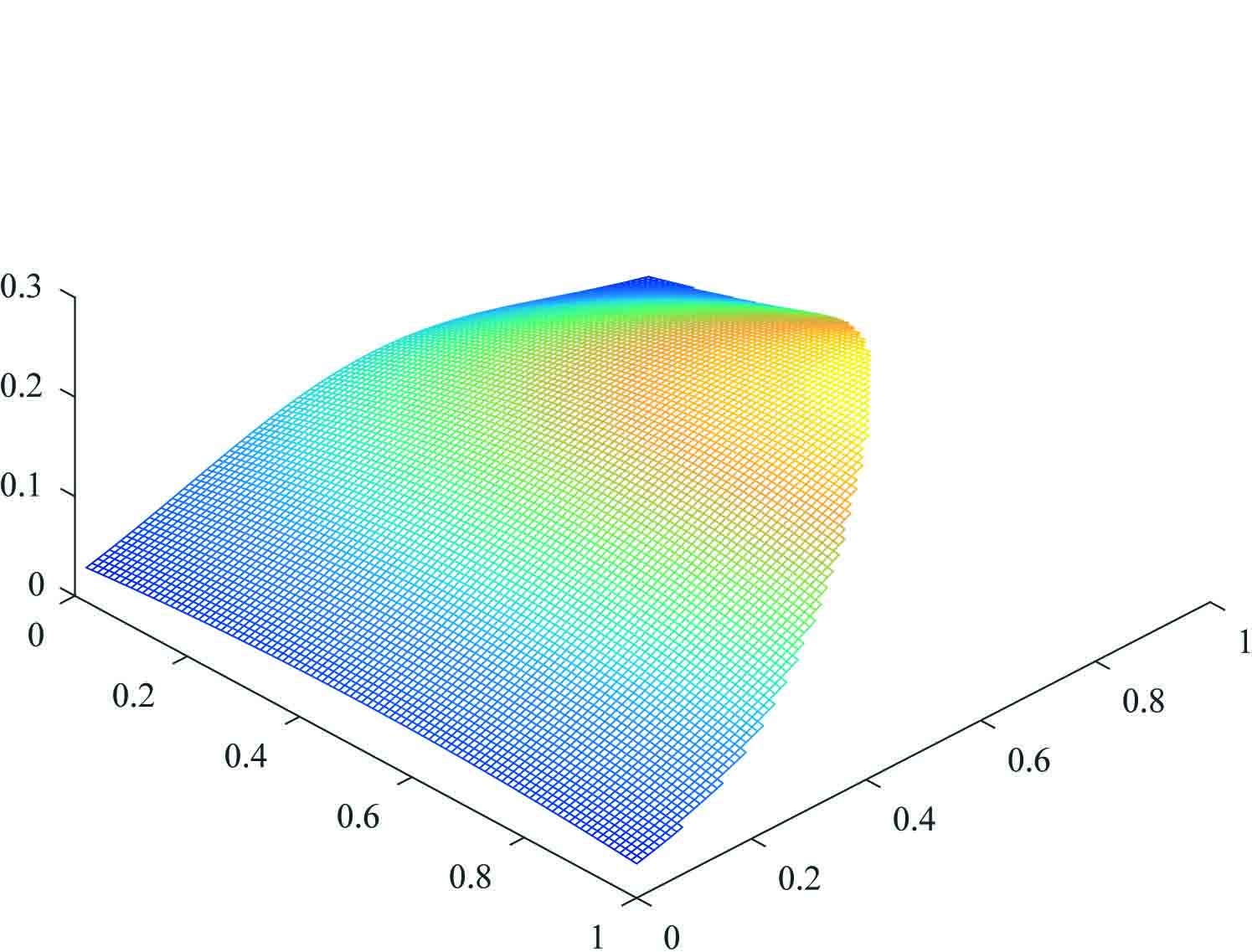}%
\\
$Q_{m}^{x}F$%
\end{center}}}
\  \  \  \  \  \  \  \
{\parbox[b]{1.6795in}{\begin{center}
\includegraphics[
height=1.2531in,
width=1.6795in
]%
{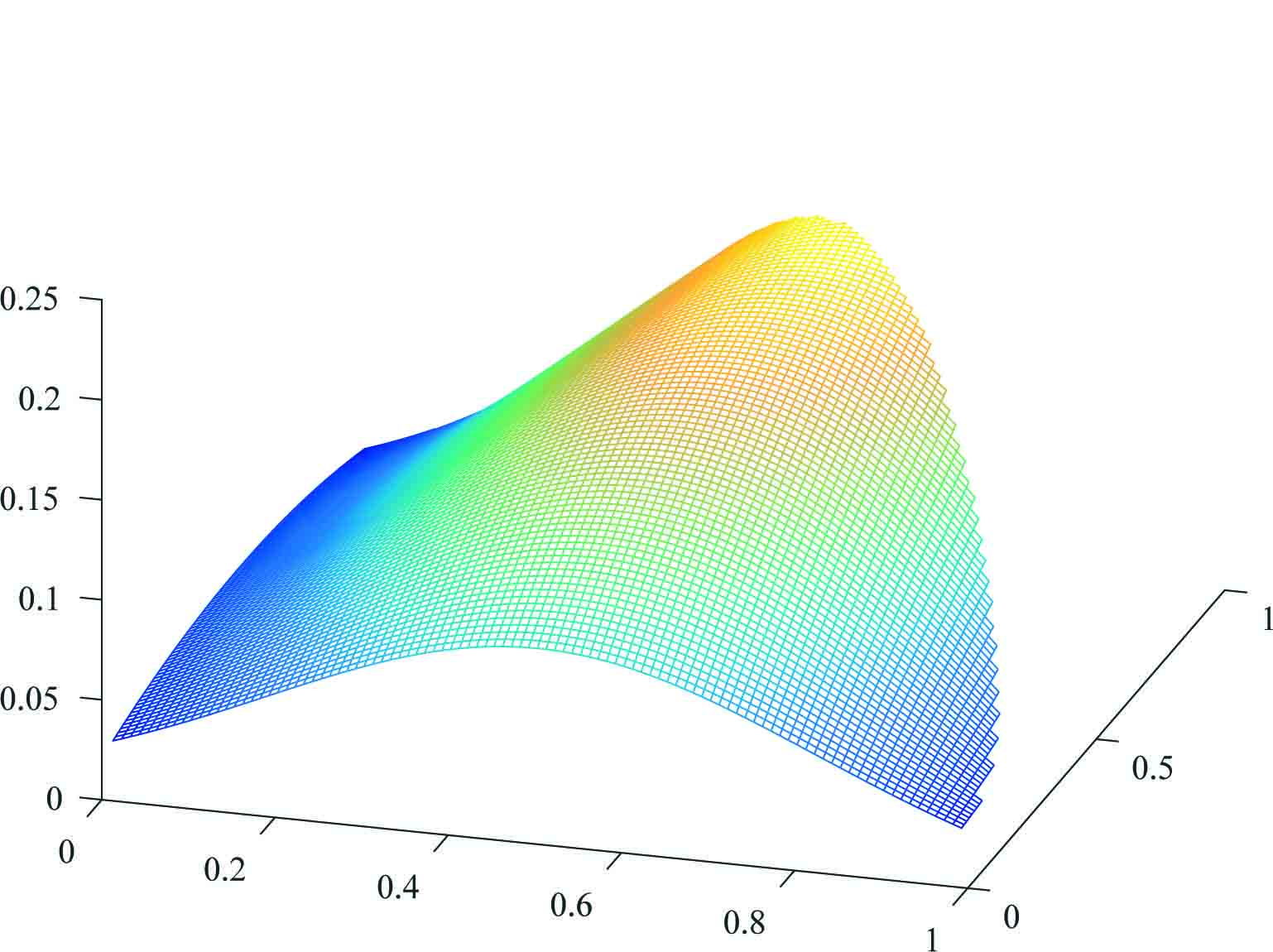}%
\\
$Q_{n}^{y}F$%
\end{center}}}
%

{\parbox[b]{1.817in}{\begin{center}
\includegraphics[
height=1.3612in,
width=1.817in
]%
{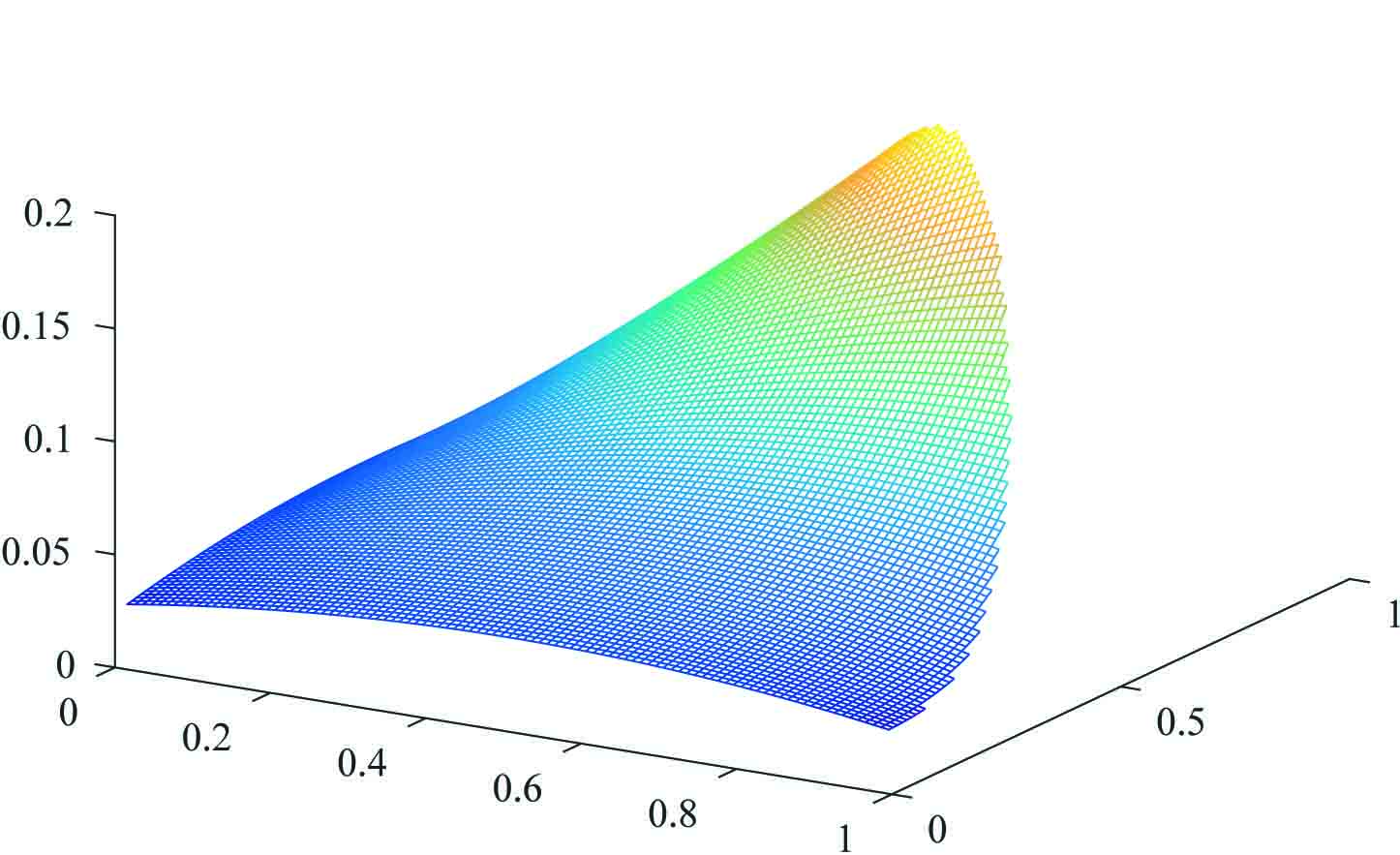}%
\\
$P_{mn}^{1}F$%
\end{center}}}
\  \  \  \  \  \  \  \
{\parbox[b]{1.8273in}{\begin{center}
\includegraphics[
height=1.3612in,
width=1.8273in
]%
{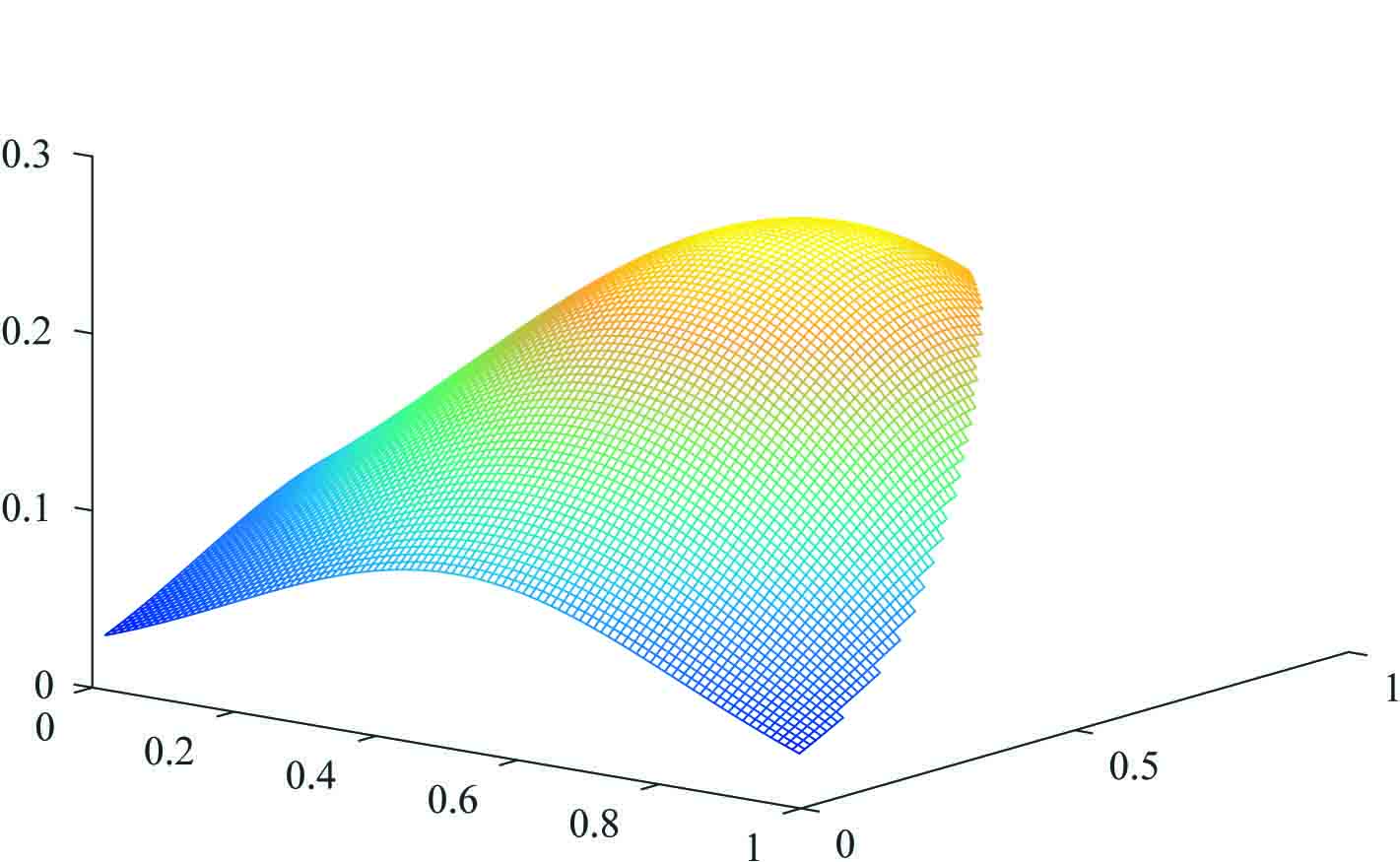}%
\\
$S_{mn}^{1}F$%
\end{center}}}

\  \  \  \  \  \  \  \  \  \  \  \  \  \  \  \  \  \  \quad Figure 2: The Cheney-Sharma
approximants for $\tilde{T}_{h}.$

\end{document}